\documentclass[a4paper,12pt,times,numbered]{amsart}
\usepackage[dvipsnames]{color}
\usepackage{amsfonts,amssymb,amsmath,amscd,amstext}
\usepackage{amsthm}
\usepackage[colorlinks=true,linkcolor=teal,citecolor=purple]{hyperref}
\usepackage[utf8]{inputenc}
\usepackage{graphicx}
\usepackage{changes}
\usepackage{comment}
\usepackage{mathtools}
\usepackage[noabbrev,capitalize,nameinlink]{cleveref}
\usepackage{mathdots}
\usepackage[a4paper,top=2.5cm,bottom=2.5cm,left=1.9cm,right=1.9cm]{geometry}
\renewcommand{\leq}{\leqslant}

\newcommand{\Rm}{\mathbb{R}^m}

\newcommand{\Rn}{\mathbb{R}^n}

\newcommand{\winf}[1]{W^{1,\infty}(#1)}

\newcommand{\X}{X}

\newcommand{\rr}{{\mathbb{R}}}

\newcommand{\Om}{\Omega}

\newcommand{\mA}{\mathcal{A}}

\newcommand{\scu}{\longrightarrow}

\newcommand{\lp}[1]{L^{p}(#1)}

\newcommand{\sol}[1]{W^{1,p}_{loc}(#1)}

\newcommand{\anso}[1]{W^{1,p}_X(#1)}

\newcommand{\ansol}[1]{W^{1,p}_{X,loc}(#1)}

\newcommand{\Cm}{\C_{P}}

\newcommand{\lul}{L^1_{loc}(\Omega)}
\newcommand{\C}{\mathcal{C}}

\newcommand{\average}{{\mathchoice {\kern1ex\vcenter{\hrule height.4pt
				width 6pt
				depth0pt} \kern-9.7pt} {\kern1ex\vcenter{\hrule height.4pt width 4.3pt
				depth0pt}
			\kern-7pt} {} {} }}

\definechangesauthor[color=red]{J}
\definecolor{champagne}{rgb}{0.97, 0.91, 0.81}

\definecolor{asparagus}{rgb}{0.53, 0.66, 0.42}

\DeclareMathOperator{\diver}{div}

\DeclareMathOperator{\im}{Im}
\newtheorem{theorem}{Theorem}[section]
\newtheorem{proposition}[theorem]{Proposition}

\newtheorem{thm}{Theorem}[section]

\theoremstyle{definition}

\newtheorem{example}[theorem]{Example}

\theoremstyle{remark}

\numberwithin{equation}{section}

\definechangesauthor[color=purple]{JP}
\definechangesauthor[color=blue]{S}

\title[Variational properties of local functionals driven by arbitrary anisotropies]{Variational properties of local functionals driven by arbitrary anisotropies}

\author[S.~Verzellesi]{Simone Verzellesi}
\address[S.~Verzellesi]{Dipartimento di Matematica, Università di Trento, via Sommarive 14, 38123 Povo (TN), Italy}
\email{simone.verzellesi@unitn.it}

\date{\today}

\thanks{
\textit{Acknowledgements}. 
The author thanks Fares Essebei, Alberto Maione, Fabio Paronetto, Andrea Pinamonti and Francesco Serra Cassano for interesting and valuable conversations on the topic of the paper.
The author is member of the Istituto Nazionale di Alta Matematica (INdAM), Gruppo Nazionale per l'Analisi Matematica, la Probabilità e le loro Applicazioni (GNAMPA).
The author has received funding from INdAM under the INdAM--GNAMPA 2023 Project \textit{Equazioni differenziali alle derivate parziali di tipo misto o dipendenti da campi di vettori}, codice CUP\_E53\-C22\-001\-930\-001. 
}

\begin{document}
\maketitle
\begin{abstract}
    We provide integral representation and $\Gamma$-compactness results for anisotropic local functionals depending on arbitrary Lipschitz continuous vector fields. In particular, neither bracket-generating assumptions nor linear independence conditions are required.
\end{abstract}
\section{Introduction}
Since its introduction in the seminal papers \cite{MR0375037,MR0448194}, 
the variational tool of \emph{$\Gamma$-convergence} has proved to be of fundamental importance in the development of modern analysis (cf. \cite{MR1968440,MR1684713,MR1201152}) and in solving problems arising from applications, including phase transitions, elasticity and fracture theory 
(cf. e.g. \cite{MR2390547,MR3489738,MR1854999,MR1916989}). A remarkable instance can be found in \cite{MR0583636,MR0839727,MR0794824,MR0567216}, where the authors studied properties of \emph{integral representation} and $\Gamma$-convergence of \emph{local functionals} defined over Euclidean functional spaces.
By integral representation one means finding conditions under which an arbitrary functional $F(u,A)$, being $u$ a function and $A$ a set, can be expressed in the integral form
\begin{equation}\label{introeucl}
    F(u,A)=\int_Af_e(x,u,Du)\,dx
\end{equation}
where the Euclidean \emph{Lagrangian} $f_e(x,u,\xi)$ typically satisfies some structural properties inherited by $F$. Integral representations as in \eqref{introeucl} are a crucial tool to deal with $\Gamma$-compactness properties, since they allow to show the closure of suitable classes of integral functionals under $\Gamma$-convergence.
Starting from \cite{MR1437714,MR1404326}, many typical problems of the calculus of variations have been transposed into the context of variational functional driven by suitable families of vector fields (cf. e.g. \cite{MR2592503,MR1381780,MR1976353,Capogna2024,MR4021977,MR1448000,MR4230967,MR4536010,MR4504133,MR4054935,MR1865002,PVW}). The key point of this generalization consists in defining a degenerate notion of \emph{$ X$-gradient} $Xu$ 
starting from a family of Lipschitz continuous vector fields $ \X=(X_1,\ldots,X_m)$, with $m\leq n$, defined on an open bounded set $\Om\subseteq\rr^n$.
This construction is rather general, and encompasses the case of \emph{Riemannian manifolds} (cf. \cite{MR1138207}), \emph{Carnot groups} (cf. \cite{MR2363343}), \emph{sub-Riemannian manifolds} (cf. \cite{MR3971262}) and \emph{Carnot-Carathéodory spaces} (cf. \cite{MR1421823}). 
The $X$-gradient $Xu$, which plays the role of the Euclidean gradient $D u$, allows to propose a functional framework suitable for the problems of calculus of variations, with the introduction of the functional spaces $W^{1,p}_X(\Om)$ and $BV_X(\Om)$ (cf. \cite{MR1437714}). 
Recently, the authors of \cite{MR4609808,MR4566142,MR4108409} generalized the results of \cite{MR0583636,MR0839727,MR0794824} to this anisotropic setting. More precisely, they studied integral representation properties of the form
\begin{equation}\label{introin}
    F(u,A)=\int_Af(x,u,Xu)\,dx
\end{equation}
under various assumptions,
together with $\Gamma$-compactness properties for sequences of integral functionals as in \eqref{introin}. We refer to \cite{MR3297726,MR3707084,MR4355973} for similar results in a Cheeger-Sobolev metric setting. All the results in \cite{MR4609808,MR4566142,MR4108409} have been obtained under a structural assumption on $\X$, the so-called \emph{linear independence condition (LIC)} (cf. \cite[Definition 1.1]{MR4108409}). More precisely, $\X$ satisfies the linear independence condition if
\begin{equation}\label{lic}
\tag{LIC}
    \text{$X_1(x),\ldots,X_m(x)$ are linearly independent for a.e. $x\in \Om$.}
\end{equation}
In particular, the approach of \cite{MR4609808,MR4566142,MR4108409} consists in applying the results in \cite{MR0583636,MR0839727,MR0794824} to obtain Euclidean integral representations as in \eqref{introeucl}. That being done, \eqref{lic} plays a crucial role to upgrade the Euclidean representation \eqref{introeucl} to a suitable anisotropic representation as in \eqref{introin}. More precisely, since \eqref{introeucl} gives rise to a representation depending on an Euclidean Lagrangian $f_e(x,u,\xi)$, the authors of \cite{MR4609808,MR4566142,MR4108409} exploits \eqref{lic} to define a new anisotropic Lagrangian $f(x,u,\eta)$ in such a way that
\begin{equation}\label{introeuclanisotr}
    f_e(x,u,Du)=f(x,u,Xu)
\end{equation}
for any sufficiently regular function $u$. Further to \cite{MR4609808,MR4566142,MR4108409}, an interesting open question was whether these results could be generalised beyond the \eqref{lic} setting.\\

In this paper, we provide an affirmative answer to the above issue, showing that all the results in \cite{MR4609808,MR4566142,MR4108409} still hold even without requiring \eqref{lic}. The value of this result is at least twofold. On the one hand, avoiding \eqref{lic} allows to consider anisotropies in the greatest generality. In particular, our results apply to the whole sub-Riemannian framework of Carnot-Carathéodory spaces. Indeed, while \eqref{lic} is general enough to cover many relevant settings, among which Carnot groups and \emph{Grushin spaces} (cf. \cite{MR4108409}), it is easy to provide instances of Carnot-Carathéodory spaces whose associated generating vector fields do not satisfy \eqref{lic} (cf. \Cref{excc}). On the other hand, our generality allows to consider the case in which a fixed family $\X$ is replaced by a sequence of families $(\X^h)_h$ converging to a limiting family $\X$ in any reasonable sense. Even assuming that each family $\X^h$ satisfies \eqref{lic}, not even the strongest convergence (say, for instance, $C^\infty$) can guarantee in general that $\X$ will do the same (cf. \Cref{exseq}). This last consideration takes on concrete relevance in matters of $\Gamma$-compactness, and will be the subject of further research.\\

Our approach starts from noticing that \eqref{introeuclanisotr} is essentially the only point where \eqref{lic} proves fundamental in the approach of \cite{MR4609808,MR4566142,MR4108409}. Therefore, the crucial part of this work is to achieve \eqref{introeuclanisotr} without requiring \eqref{lic}. To explain our approach, we should start by recalling the strategy of \cite{MR4609808,MR4566142,MR4108409}. To this aim, fix a point $x\in\Om$ and assume that $X_1(x),\ldots,X_m(x)$ are linearly independent in $\rr^m$. This implies that the projection map $\C(x):\rr^n\longrightarrow\rr^m$ induced by $X_1(x),\ldots,X_m(x)$ is surjective. In this case, it is possible to set
\begin{equation}\label{introanisofrancesco}
    f(x,u,\eta)=f_e\left(x,u, \C(x)^{-1}(\eta)\right),
\end{equation}
being $ \C(x)^{-1}$ a suitable right-inverse map of $\C(x)$, and to show, under additional assumptions on $f_e$, that \eqref{introanisofrancesco} suffices to infer \eqref{introeuclanisotr}. Since our vector fields may be in general linearly dependent, $\C(x)$ may not be right-invertible. To this aim, we replace $ \C(x)^{-1}$ with the so-called \emph{Moore-Penrose pseudo-inverse} of $\C(x)$ (cf. \cite{MR1417720}), say $\Cm(x)$, and we set
\begin{equation}\label{introanisomia}
    f(x,u,\eta)=f_e\left(x,u, \Cm(x)\cdot\eta\right).
\end{equation}
A careful analysis of the properties of $\Cm(x)$ (cf. \Cref{alglemma}) will allow us to exploit \eqref{introanisomia} to provide anisotropic representations as in \eqref{introeuclanisotr} (cf. \Cref{anisorappr}). Once \eqref{introeuclanisotr} is achieved, we devote the rest of the paper to the generalization of the results in \cite{MR4609808,MR4566142,MR4108409} (cf. \Cref{igegcproofs} and \Cref{furthersection}). We decided to make this last part of the exposition as concise as possible, both to emphasise the crucial importance of \Cref{anisorappr}, and because, once \Cref{anisorappr} is obtained, the proofs work exactly like their counterparts in \cite{MR4609808,MR4566142,MR4108409}. We stress that our results are substantially analogous to those proved in \cite{MR4609808,MR4566142,MR4108409}. Nevertheless, the possible non-surjectivity of $\C(x)$ brings out some interesting new phenomena. First of all, a deep look at the shape of $f$ in \eqref{introanisomia} reveals that it is constant outside the range of $\C(x)$ (cf. \Cref{anisorappr}). More precisely, if we orthogonally decompose any $\eta\in\rr^m$ as
\begin{equation*}
    \eta=\C(x)\cdot\xi_\eta+\eta^\perp
\end{equation*}
for some $\xi_\eta\in\rr^n$, then $f$ satisfies
\begin{equation}\label{introcost}
    f(x,u,\eta)=f(x,u,\C(x)\cdot\xi_\eta).
\end{equation}
Anyway, \eqref{introcost} is verified only if $f$ is defined as in \eqref{introanisomia}. Indeed, it is possible to provide integral representations as in \eqref{introin} by arbitrarily choosing the value of the corresponding anisotropic Lagrangian outside the range of $\C(x)$ (cf. \Cref{exsect4} and \Cref{sufficiente}). Notwithstanding, we prove that \eqref{introcost} is a sufficient property to guarantee  uniqueness in the integral representation (cf. \Cref{2411modiofmythm24}). Another consequence of \eqref{introcost} is that $f$ as in \eqref{introanisomia} cannot inherit from $f_e$ full coercivity in the gradient argument. Nevertheless, one can easily observe how the structural properties of an integral functional depend, in our case, solely on the behaviour of the Lagrangian on the range of $\C(x)$ (cf. \Cref{sufficiente}).\\

The paper is organized as follows. In \Cref{prelisec} we collect some basic preliminaries and we provide some detailed motivations. In \Cref{represect} we prove the core results of this paper, namely \Cref{alglemma} and \Cref{anisorappr}, showing \eqref{introeuclanisotr}. In \Cref{igegcproofs} we show how to exploit \Cref{anisorappr} to provide integral representation (cf. \Cref{2411modiofmythm24}) and $\Gamma$-compactness (cf. \Cref{gammathmwithproof}) for a particular class of local functionals. Moreover, we motivate the content of the statements and we show why our setting of conditions is the optimal one (cf. \Cref{exsect4} and \Cref{sufficiente}). Finally, in \Cref{furthersection} we state, for further references, the remaining counterparts of the results in \cite{MR4609808,MR4566142,MR4108409}.
\section{Preliminaries and motivations}\label{prelisec}
\subsection{Main notation}
In the following, we fix $m,n\in\mathbb{N}$ such that $0<m\leq n$. For $\alpha,\beta\in\mathbb{N}\setminus\{0\}$, we denote by $M(\alpha,\beta)$ the set of matrices with $\alpha$ rows and $\beta$ columns. If $\alpha,\beta$ are as above and $L:\rr^\alpha\longrightarrow\rr^\beta$ is a linear map, we denote by $\ker(L)\subseteq\rr^\alpha$ and $\im(L)\subseteq\rr^\beta$ respectively its kernel and its range. We fix an open bounded set $\Om\subseteq\rr^n$, and we denote by $\mA$ the class of all the open subsets of $\Om$. In the following, we mean vectors in $\rr^\alpha$ as matrices in $M(\alpha,1)$.
\subsection{Lipschitz continuous vector fields} Given a family $\X:=(X_1,\ldots,X_m)$ of Lipschitz continuous vector fields on $\Omega$, we denote by $\C(x)$ the $m\times n$ matrix defined by \[\C(x):=[c_{j,i}(x)]_{\substack{{i=1,\dots,n}\\{j=1,\dots,
m}}}\,,\]
where $c_{j,i}$ is a Lipschitz continuous function on $\Om$ for any $j=1,\ldots,m$ and any $i=1,\ldots,n$ and
$$X_j=\sum_{i=1}^nc_{j,i}\frac{\partial}{\partial x_i}$$
for any $j=1,\ldots,m$. We recall (cf. \cite{MR1437714}) that the  \emph{$X$-gradient} $Xu$ of $u\in\lul$ is the distribution defined by 
\begin{equation*}
    Xu(\varphi)=-\int_\Om u \diver (\varphi(x)\cdot \C(x)){\rm d} x 
\end{equation*}
 for any $\varphi\in C^\infty_c(\Om,\Rm)$.
The associated anisotropic Sobolev spaces are defined by
	$$\anso{\Om}=\{u\in L^p(\Om)\,:\,Xu\in L^p(\Om,\rr^m)\}\quad\text{and}\quad\ansol{\Om}=\bigcap\left\{\anso{A}\,:\,A\in\mA,\,A\Subset\Om\right\}.$$
It is well-known (cf. \cite{FS}) that the vector space $\anso{\Om}$, endowed with the norm
$$\Vert u\Vert_{\anso{\Om}}:=\Vert u\Vert_{L^p(\Om)}+\Vert Xu\Vert_{L^p(\Om,\Rm)},$$
is a Banach space for any $1\leq p<+\infty$, and that it is reflexive when $1<p<+\infty$. Under the Lipschitz continuity assumption on $\X$, $W^{1,p}(\Om)$ embeds continuously into $W^{1,p}_X(\Om)$ (cf. \cite{MR4108409}), and the inclusion may be strict (cf. \cite{MR4609808}). More precisely, if $u\in \sol{\Om}$, its $X$-gradient admits the Euclidean representation
\begin{equation}\label{euclrapprgrad}
    Xu(x)=\C(x)\cdot Du(x)
\end{equation}
for a.e. $x\in\Om$.
\subsection{Relevant vector fields}\label{relevantsection}
As already known, many relevant families of vector fields can already be found when \eqref{lic} holds, such as the Euclidean space, Carnot groups and Grushin spaces (cf. \cite{MR4108409}). Nevertheless, avoiding \eqref{lic} is crucial to ensure that the results of \cite{MR4609808,MR4566142,MR4108409}, gain sufficient generality to be applied, for example, to the Carnot-Carathéodory setting. 
\begin{example}\label{excc}
     As an instance, consider the the family $X=(X_1,X_2)$ of vector fields defined on $\Om=(-1,1)^2\subseteq\rr^2$ by
    \begin{equation*}
        X_1(x)=\frac{\partial}{\partial x_1}\qquad\text{and}\qquad X_2(x)=
\begin{cases}
0&\text{ if }x_1\in(-1,0)\\
x_1\frac{\partial}{\partial x_2}&\text{ if }x_1\in[0,1)
\end{cases}
\,. 
    \end{equation*}
    for any $x=(x_1,x_2)\in\Om$. Is is easy to check that $\Om$, endowed with the control distance induced by $X$ (cf. \cite{MR0793239}), is a Carnot-Carathéodory space. Moreover, $X_1,X_2$ are Lipschitz continuous on $\Om$. Nevertheless, they do not satisfy \eqref{lic}. 
\end{example}
Moreover, as pointed out in the introduction, \eqref{lic} may not in general be preserved under even strong notions of convergence. 
\begin{example}\label{exseq}
    For any $h\in\mathbb N\setminus\{0\}$, consider the the family $\X^h=(X_1,X^h_2)$ of vector fields defined on $\rr^2$ by
    \begin{equation*}
        X_1(x)=\frac{\partial}{\partial x_1}\qquad \text{and}\qquad X^h_2(x)=\frac{1}{h}\frac{\partial}{\partial x_2}
    \end{equation*}
    for any $x=(x_1,x_2)\in\rr^2$. For any $h$ as above, $\X^h$ is made of smooth and globally Lipschitz continuous vector fields which satisfy \eqref{lic} on $\rr^2$. Nevertheless, $(\X^h)_h$ convergence uniformly, with all its derivatives, to $\X=(X_1,0)$, which clearly does not satisfy \eqref{lic}.
\end{example}

\subsection{$\Gamma$-convergence and local functionals}
For a complete account to $\Gamma$-convergence, we refer the reader to \cite{MR1968440,MR1684713,MR1201152}. We just recall that if $(\mathcal X,\tau)$ is a first-countable topological space, a sequence of functionals $(F_h)_h:\mathcal X\longrightarrow [0,+\infty]$ is said to \emph{$\Gamma(\tau)$-converge} to a functional $F:\mathcal X\longrightarrow [0,+\infty]$ if the following two conditions hold.
\begin{itemize}
    \item For any $u\in\mathcal X$ and any sequence $(u_h)_h$ converging to $u$, then
    \begin{equation}\label{liminfineq}
        F(u)\leq\liminf_{h\to+\infty}F_h(u_h).
    \end{equation}
    \item For any $u\in\mathcal X$, there exists a sequence $(u_h)_h$ converging to $u$ such that
    \begin{equation}\label{recodef}
        F(u)=\lim_{h\to+\infty}F_h(u_h).
    \end{equation}
\end{itemize}
Sequences for which \eqref{recodef} holds are known as \emph{recovery sequences}. We conclude this preliminary section with a list of definitions concerning local functions, in order to keep the discussion as self-contained as possible. Since an exhaustive treatment of this topic goes beyond the scope of this paper, we refer to \cite{MR1201152} for further references, and to \cite{MR4609808,MR4566142,MR4108409} for the anisotropic setting notation. If $F:\lp{\Om}\times\mA\scu[0,+\infty]$ (resp. $F:\anso{\Om}\times\mA\scu[0,+\infty]$), we say that $F$ is:
	\begin{itemize}
	\item  a \emph{measure} if $F(u,\cdot)$ is a measure for any $u\in\lp{\Om}$ (resp. $u\in\anso{\Om}$);
	\item  \emph{local} if, for any $A\in\mA$ and $u,v\in\lp{\Om}$ (resp. $u,v\in\anso{\Om}$), then
	$$u|_{A}=v|_{A}\implies F(u,A)=F(v,A);$$
\item  \emph{convex} on $\anso{\Om}$ if $F(\cdot,A)$ restricted to $\anso{\Om}$ is convex for any $A\in\mA$;
	\item  \emph{$L^p$-lower semicontinuous} (resp. \emph{$W_X^{1,p}$-lower semicontinuous}) if $F(\cdot,A)$ is  $L^p$-lower semicontinuous (resp. $W_X^{1,p}$-lower semicontinuous) for any $A\in\mA$;
	\item \emph{weakly*- sequentially lower semicontinuous} if $F(\cdot,A)$ restricted to ${\winf{\Om}}$ is sequentially lower semicontinuous with respect to the weak*- topology of $\winf{\Om}$ for any $A\in\mA$.
	\end{itemize}

\section{Anisotropic representation of Euclidean Lagrangians}\label{represect}
This section constitutes the core of this paper. More precisely, we show how to express a Euclidean Lagrangian in terms of an anisotropic Lagrangian, proving \eqref{introeuclanisotr}. 
\subsection{Algebraic properties of the Moore-Penrose pseudo-inverse}\label{pseudosection}
 Following the notation of \cite{MR4609808,MR4108409}, for any $x\in\Om$ we define the linear map $\C(x):\Rn\scu\Rm$ by
	\begin{equation*}\label{Lx}
		\C(x)(\xi)=\C(x)\cdot\xi
	\end{equation*}
	for any $\xi\in\rr^n$. Moreover, we let
\begin{equation*}\label{N&Vx}
		N_x=\ker(\C(x))\qquad\text{and}\qquad V_x=\left\{\C(x)^T \cdot\eta\ :\ \eta\in\mathbb{R}^m\right\}.
	\end{equation*}
	From standard linear algebra (cf. e.g. \cite{MR0575349}), we know that $\Rn=N_x\oplus V_x$. Hence, for any $x\in\Om$ and $\xi\in\Rn$, there are uniqe $\xi_{N_x}\in N_x$ and $\xi_{V_x}\in V_x$ such that
	\begin{equation}\label{splitting}
		\xi=\xi_{N_x}+\,\xi_{V_x}.
	\end{equation}
Therefore, the map $\Pi_x:\Rn\to V_x$ defined by $\Pi_x(\xi)=\xi_{V_x}$ is well-posed. The authors of \cite{MR4609808,MR4108409} exploited in a crucial way \eqref{lic} to ensure the existence of a right-inverse map associated to $\C(x)$. Precisely, if $X_1(x),\ldots,X_m(x)$ are linearly independent at some $x\in\Om$, then any $\eta\in\rr^m$ can be expressed in the form $\eta=\C(x)\cdot\xi_\eta$ for some $\xi_\eta\in\rr^n$. In the general case, we decompose $\eta\in\rr^m$ as
\begin{equation*}
    \eta=\C(x)\cdot\xi_\eta+\eta^\perp,
\end{equation*}
where $\eta^\perp\in \im(\C(x))^\perp.$ We stress that $\xi_\eta$ is uniquely defined only  modulo $\ker(\C(x))$.
Since $\C(x)$ may not have full rank, our approach must therefore differ from \cite{MR4609808,MR4108409}.
    Let $\Cm:\Om\longrightarrow M(n,m)$ be defined so that $\Cm(x)$ is the Moore-Penrose pseudo-inverse of $\C(x)$ (cf. \cite{MR1417720}) for any $x\in\Om$. Precisely, for a fixed $x\in\Om$, $\Cm(x)$ is the unique matrix in $M(n,m)$ such that (cf. \cite{MR1417720})
    \begin{equation} \label{pseudoinvproperties}
    \begin{aligned}
        &\Cm(x)\cdot\C(x)\cdot\Cm(x)=\Cm(x),\qquad \C(x)\cdot\Cm(x)\cdot\C(x)=\C(x),\\
        &\Cm(x)\cdot\C(x)=\C(x)^T\cdot\Cm(x)^T,\qquad \C(x)\cdot\Cm(x)=\Cm(x)^T\cdot\C(x)^T.
    \end{aligned}
    \end{equation}
Our anisotropic representation is based on the following properties of $\Cm$.
\begin{proposition}\label{alglemma}
    Let $\Cm$ be the above-defined map. Moreover, for any $x\in\Om$, let $\Cm(x):\rr^m\longrightarrow \rr^n$ be the linear map defined by 
    \begin{equation*}
        \Cm(x)(\eta)=\Cm(x)\cdot\eta
    \end{equation*}
    for any $\eta$. Then the map
    \begin{equation*}
        x\mapsto\Cm(x)(\eta)
    \end{equation*}
    is measurable for any $\eta\in\rr^m$. Moreover, for any $x\in\Om$, the following facts hold.
    \begin{itemize}
        \item [(i)] $\im(\Cm(x))= V_x$.
        \item[(ii)] $\Pi_x(\xi)=\Cm(x)\cdot\C(x)\cdot\xi$ for any $\xi\in\rr^n$. 
        \item [(iii)] $\ker(\Cm(x))=\im(\C(x))^\perp$.
    \end{itemize}
\end{proposition}
\begin{proof}
    For a given $\eta\in\rr^n$, it is well-known (cf. \cite{MR1417720}) that 
    \begin{equation*}
        \Cm(x)\cdot\eta=\lim_{h\to+\infty}\left(\C(x)^T\cdot\C(x)+\frac{1}{h}I_n\right)^{-1}\cdot\C(x)^T\cdot\eta.
    \end{equation*}
    for any $x\in\Om$. In particular, being $\C$ continuous over $\Om$, $x\mapsto\Cm(x)\cdot\eta$ is the pointwise limit of continuous functions, and hence it is measurable. Now we fix $x\in\Om$. Notice that, by \eqref{pseudoinvproperties},
    \begin{equation*}
        \Cm(x)\cdot\eta=\Cm(x)\cdot\C(x)\cdot\Cm(x)\cdot\eta=\C(x)^T\cdot\left(\Cm(x)^T\cdot\Cm(x)\cdot\eta\right)
    \end{equation*}
    for any $\eta\in\rr^m$, so that $\im(\Cm(x))\subseteq V_x$. To prove the other inclusion, it suffices to show (ii). To this aim, fix $\xi\in\rr^n$. by \eqref{splitting} and \eqref{pseudoinvproperties},
    \begin{equation*}
\C(x)\cdot\Cm(x)\cdot\C(x)\cdot\xi=\C(x)\cdot\xi=\C(x)\cdot\left(\Pi_x(\xi)+\xi_{N_x}\right)=\C(x)\cdot\Pi_x(\xi).
    \end{equation*}
   Since we already know that $\Cm(x)\cdot\C(x)\cdot\xi\in V_x$, and being $\C(x)$ injective on $V_x$, (ii) follows. To prove (iii), fix $\eta\in\ker(\Cm(x))$ and $\xi\in\rr^n$. Then, by \eqref{pseudoinvproperties},
   \begin{equation*}
       \begin{split}
\eta^T\cdot\C(x)\cdot\xi=\eta^T\cdot\C(x)\cdot\Cm(x)\cdot\C(x)\cdot\xi=\left(\Cm(x)\cdot\eta\right)^T\cdot\C(x)^T\cdot\C(x)\cdot\xi=0,
       \end{split}
   \end{equation*}
   so that $\eta\in\im(\C(x))^\perp$. Hence $\ker(\Cm(x))\subseteq\im(\C(x))^\perp$.  Assume by contradiction that there exists $\eta\neq 0$ such that $\eta\in\im(\C(x))^\perp\cap\ker(\C(x))^\perp$. In view of \eqref{pseudoinvproperties},
   \begin{equation}\label{cpinj}
       \Cm(x)\cdot\eta=\Cm(x)\cdot\C(x)\cdot\Cm(x)\cdot\eta.
   \end{equation}
   Since we know that $\ker(\Cm(x))\subseteq\im(\C(x))^\perp$, then $\im(\C(x))\subseteq\ker(\Cm(x))^\perp$, so that both $\eta$ and $\C(x)\cdot\Cm(x)\cdot\eta$ belongs to $\ker(\Cm(x))^\perp$. Being $\Cm(x)$ injective on $\ker(\Cm(x))^\perp$, we conclude from \eqref{cpinj} that $\eta=\C(x)\cdot\Cm(x)\cdot\eta$, a contradiction with $\eta\in\im(\C(x))^\perp$.
\end{proof}

\subsection{The anisotropic representation result}
We exploit \Cref{alglemma} to show that the anisotropic Lagrangian in \eqref{introanisomia} satisfies \eqref{introeuclanisotr}.
\begin{proposition}\label{anisorappr}
    Let $f_e:\Om\times\rr\times\rr^n\longrightarrow[0,+\infty]$ be a Carathéodory function.  Assume that 
\begin{equation}\label{0702b}
        f_e(x,u,\xi)=f_e(x,u,\Pi_x(\xi))
    \end{equation}
    for a.e. $x\in\Om$, any $u\in\rr$ and any $\xi\in\rr^n$. Define the map $f:\Om\times\rr\times\rr^m\longrightarrow [0,+\infty]$ by
    \begin{equation}\label{fanisotrdef}
        f(x,u,\eta)=f_e(x,u,\Cm(x)\cdot\eta)
    \end{equation}
    for any $x\in\Om$, any $u\in\rr$ and any $\eta\in\rr^m$. Then $f$ is a Carathéodory function such that
    \begin{equation}\label{costonker}
        f(x,u,\eta)=f(x,u,\C(x)\cdot\xi_\eta)
    \end{equation}
    and
    \begin{equation}\label{0702h}
        f_e(x,u,\xi)=f(x,u,\C(x)\cdot\xi)
    \end{equation}
    for a.e. $x\in\Om$, any $u\in\rr$, any $\eta\in\rr^m$ and any $\xi\in\rr^n$. Moreover, $f$ enjoys the following properties.
    \begin{itemize}
        \item [(i)] If there exist $a\in L^1_{loc}(\Om)$ and some $b,c\geq 0$ such that
    \begin{equation}\label{0702g}
     f_e(x,u,\xi)\leq a(x)+b|u|^p+c|\C(x)\cdot\xi|^p
    \end{equation}
    for a.e. $x\in\Om$, any $u\in\rr$ and any $\xi\in\rr^n$, then
    \begin{equation}\label{0702i}
        f(x,u,\C(x)\cdot\xi)\leq a(x)+b|u|^p+ c|\C(x)\cdot\xi|^p
    \end{equation}
     for a.e. $x\in\Om$, any $u\in\rr$ and any $\xi\in\rr^n$.
     \vspace{3pt}
        \item [(ii)] If there exist $d>0$ such that
    \begin{equation}\label{0702g2}
     d|\C(x)\cdot\xi|^p\leq f_e(x,u,\xi)
    \end{equation}
    for a.e. $x\in\Om$, any $u\in\rr$ and any $\xi\in\rr^n$, then
    \begin{equation}\label{0702i2}
        d|\C(x)\cdot\xi|^p\leq f(x,u,\C(x)\cdot\xi)
    \end{equation}
     for a.e. $x\in\Om$, any $u\in\rr$ and any $\eta\in\rr^m$.
     \vspace{3pt}
     \item [(iii)] If  $f_e(x,u,\xi)=f_e(x,\xi)$ for a.e. $x\in\Om$, any $u\in\rr$ and any $\xi\in\rr^n$, then 
    $f(x,u,\eta)=f(x,\eta)$ for a.e. $x\in\Om$, any $u\in\rr$ and any $\eta\in\rr^m$.
     \vspace{3pt}
     \item [(iv)] If 
     \begin{equation}\label{feuclconv}
         f_e(x,u,\cdot)\text{ is convex}
     \end{equation}
     for a.e. $x\in\Om$ and any $u\in\rr$, then
     \begin{equation}\label{0802conv2}
         f(x,u,\cdot)\text{ is convex}
     \end{equation}
     for a.e. $x\in\Om$ and any $u\in\rr$.
     \vspace{3pt}
     \item [(v)] If $$f_e(x,\cdot,\cdot)\text{ is convex}$$ for a.e. $x\in\Om$, then 
     \begin{equation}\label{0802conv1}
         f(x,\cdot,\cdot)\text{ is convex}
     \end{equation}
     for a.e. $x\in\Om$.
    \end{itemize}
\end{proposition}
\begin{proof}
 Let $f$ be the function in \eqref{fanisotrdef}. First we show that $f$ is a Carathéodory function. To this aim, fix $u\in\rr$ and $\eta\in\rr^m$, and define the function $\Phi_{u,\eta}:\Om\longrightarrow\rr\times\rr^n$ by $$\Phi_{u,\eta}(x)=(u,\Cm(x)\cdot\eta)$$
    for any $x\in\Om$. Being $x\mapsto\Cm(x)\cdot\eta$ measurable by \Cref{alglemma}, then $\Phi_{u,\eta}$ is measurable. Since 
    \begin{equation}\label{fiscara}
f(x,u,\eta)=f_e(x,\Phi_{u,\eta}(x))
    \end{equation}
    for a.e. $x\in\Om$, and being $f_e$ a Carathéodory function, we deduce from \cite[Proposition 3.7]{MR0990890} that $x\mapsto f(x,u,\eta)$ is measurable for any $u\in\rr$ and any $\eta\in\rr^m$. Fix now $x\in\Om$ and define $\Psi_x:\rr\times\rr^m\longrightarrow\rr\times\rr^n$ by
    $$\Psi_x(u,\eta)=\Phi_{u,\eta}(x)$$
    for any $u\in\rr$ and any $\eta\in\rr^m$. Clearly, $\Psi_x$ is a linear function. In particular, by \eqref{fiscara} and being $f_e$ a Carathéodory function, then $(u,\eta)\mapsto f(x,u,\eta)$ is continuous for a.e. $x\in\Om$, so that $f$ is a Carathéodory function. Moreover, in view of \eqref{fiscara}, the linearity of $\Psi_x$ and the definition of $f$, (iii), (iv) and (v) easily follows. Moreover, \eqref{costonker} follows directly from (iii) of \Cref{alglemma}. Let us prove \eqref{0702h}.
    In view of (ii) of \Cref{alglemma}, \eqref{0702b} and the definition of $f$, we infer that
    \begin{equation*}
        \begin{split}
            f(x,u,\C(x)\cdot\xi)=f_e(x,u,\Cm(x)\cdot\C(x)\cdot\xi)
            =f_e(x,u,\Pi_x(\xi))
            =f_e(x,u,\xi)
        \end{split}
    \end{equation*}
    for a.e. $x\in\Om$, any $u\in\rr$ and any $\xi\in\rr^n$, so that \eqref{0702h} follows. Finally, (i) and (ii) are direct consequences of \eqref{0702h}. 
\end{proof}
\section{Integral representation and $\Gamma$-compactness without \eqref{lic}: proofs in a prototypical example}\label{igegcproofs}
In this section we prove integral representation and $\Gamma$-compactness results in the setting of \emph{translation-invariant} local functionals proposed in \cite{MR0794824,MR4108409}. As already pointed out, our proofs will be concise and focused on the application of \Cref{anisorappr}.
\subsection{Integral representation} Let us begin with the generalization of \cite[Theorem 3.12]{MR4108409} to our general setting.
\begin{thm}\label{2411modiofmythm24}
	Let $p\in [1,+\infty)$. Let $F:L^p(\Om)\times\mathcal{A}\longrightarrow [0,+\infty]$ satisfy the following properties.
	\begin{itemize}
		\item [(i)]$F$ is a measure.
		\item [(ii)]$F$ is local.
		\item [(iii)]$F$ is $L^p$-lower semicontinuous.
  \item [(iv)] $F(u+k,A)=F(u,A)$ for any $A\in\mA$, any $u\in C^\infty(A)$ and any $k\in\rr$.
		\item [(v)] There exist $a\in L^1_{loc}(\Om)$ and $c\geq 0$ such that
  \begin{equation*}
      F(u,A)\leq\int_Aa(x)+c|Xu|^p\,dx
  \end{equation*}
  for any $A\in\mA$ and any $u\in C^\infty(A)\cap L^p(\Om)$.
	\end{itemize}
	Then there exists a Carathéodory function $f:\Omega\times\mathbb{R}^m\longrightarrow [0,+\infty)$ 
	such that	
 \begin{equation}\label{primarappr0802}
     F(u,A)=\int_{A}f(x,Xu(x))\,d x
 \end{equation}
 for any $A\in\mA$ and any $u\in W^{1,p}_{X,loc}(A)\cap L^p(\Om)$.
 Moreover, $f$ satisfies \eqref{costonker}, \eqref{0702i} and \eqref{0802conv2}.
In addition, if $\tilde f:\Om\times\rr^m\longrightarrow [0,+\infty)$ is a Carathéodory function which verifies \eqref{costonker}, \eqref{0702i} and for which \eqref{primarappr0802} holds with $\tilde f$ in place of $f$, then
	\begin{equation} \label{2411uniquenessss}
		\tilde f(x,\eta)=f(x,\eta)
	\end{equation}
 for a.e. $x\in\Om$ and any $\eta\in\rr^m$. Finally, if there exists $d>0$ such that  \begin{equation}\label{functionalcoercive}
      d\int_A|Xu|^p\,dx\leq F(u,A)
  \end{equation}
  for any $A\in\mA$ and any $u\in C^\infty(A)\cap L^p(\Om)$, then $f$ satisfies \eqref{0702g2}.
\end{thm}
\begin{proof}
   Arguing \emph{verbatim} as in the first step of the proof of \cite[Theorem 3.12]{MR4108409}, our assumptions allow an Euclidean integral representation for $F$, meaning that 
\begin{equation}\label{euclprimarappr0802}
     F(u,A)=\int_{A}f_e(x,Du(x))\,d x
 \end{equation}
 for any $A\in\mA$ and any $u\in W^{1,p}_{loc}(A)$, where $f_e:\Om\times\rr^n\longrightarrow [0,+\infty)$ is a suitable Carathéodory function satisfying \eqref{0702b}, \eqref{0702g} and \eqref{feuclconv}. Therefore, by \Cref{anisorappr}, $f:\Om\times\rr^m\longrightarrow [0,+\infty)$ defined as in \eqref{fanisotrdef} is a Carathéodory function which satisfies \eqref{costonker}, \eqref{0702h}, \eqref{0702i} and \eqref{0802conv2}. Therefore, combining \eqref{euclrapprgrad}, \eqref{0702h} and \eqref{euclprimarappr0802},
\begin{equation}\label{primaanisotroreppr}
 F(u,A)=\int_{A}f(x,Xu(x))\,d x 
\end{equation}
    for any $A\in\mA$ and any $u\in C^\infty(A)$. In order to achieve \eqref{primarappr0802}, one exploits \eqref{primaanisotroreppr} to argue \emph{verbatim} as in the third step of the proof of \cite[Theorem 3.12]{MR4108409}. If \eqref{functionalcoercive} holds, arguing \emph{verbatim} as in the first step of the proof of \cite[Theorem 3.12]{MR4108409} we infer that $f_e$ satisfies \eqref{0702g2}, so that, by \Cref{anisorappr}, $f$ verifies \eqref{0702i2}. Finally, assume that there exists a Carathéodory function $\tilde f:\Om\times\rr^m\longrightarrow [0,+\infty)$ which verifies \eqref{costonker}, \eqref{0702i} and \eqref{primarappr0802}. By \eqref{0702i}, \eqref{primarappr0802} and proceeding as in the fifth step of the proof of \cite[Theorem 3.3]{MR4609808}, one infer that
	\begin{equation} \label{2411uniquenessssinproof}
		\tilde f(x,\C(x)\cdot\xi)=f(x,\C(x)\cdot\xi)
	\end{equation}
 for a.e. $x\in\Om$ and any $\xi\in\rr^n$. Since both $f$ and $\tilde f$ satisfy \eqref{costonker}, we conclude by \eqref{2411uniquenessssinproof} that
 \begin{equation*} \label{2411uniquenessssinproof}
		\tilde f(x,\eta)=\tilde f(x,\C(x)\cdot\xi_\eta)=f(x,\C(x)\cdot\xi_\eta)=f(x,\eta)
	\end{equation*}
 for a.e. $x\in\Om$ and any $\eta\in\rr^m$, so that \eqref{2411uniquenessss} follows.
\end{proof}
       We point out that the statement of \Cref{2411modiofmythm24} is sharp. On the one hand, neither in \eqref{0702i} nor in \eqref{0802conv2} it is reasonable to expect global bounds rather than partial bounds on $\im (\C(x))$. On the other hand, a uniqueness property as in \eqref{2411uniquenessss} may fails dropping \eqref{costonker}. This is to say, roughly speaking, that the structural properties of $F$ translates into structural properties of $f$ only as regards the part of $f$ acting on the image of $\C$. This fact is not surprising. Indeed, for a fixed $x\in\Om$, we already know that the action of $\C(x)$ is surjective only when $X_1(x),\ldots,X_m(x)$ are linearly independent. Since we are not assuming \eqref{lic}, this property may trivially fail in general.
\begin{example}\label{exsect4}
  As an instance, consider the the family $X=(X_1,X_2)$ of vector fields defined on $\Om=(0,1)^2\subseteq\rr^2$ by
    \begin{equation*}
        X_1(x)=X_2(x)=\frac{\partial}{\partial x_1}
    \end{equation*}
    for any $x=(x_1,x_2)\in\Om$. Clearly $X_1,X_2$ are Lipschitz continuous on $\Om$ and linearly dependent for any $x\in\Om$. The associated matrices $\C$ and $\Cm$ are respectively
    \begin{equation*}
        \C(x)=\left[ \begin{array}{cc}
	1 & 0  \\
	1 & 0  \\ 
\end{array}
\right ]\qquad\text{and}\qquad\Cm(x)=\left[ \begin{array}{cc}
	1/2 & 1/2  \\
	0 & 0  \\ 
\end{array}
\right ]
    \end{equation*}
for any $x\in\Om$. In particular,
\begin{equation}\label{immaginecontroes}
    N_x=\left\{(0,\lambda)\in\rr^2\,:\,\lambda\in\rr\right\}\qquad\text{and}\qquad \im(\C(x))=\left\{(\lambda,\lambda)\in\rr^2\,:\,\lambda\in\rr\right\}
\end{equation}
  for any $x\in\Om$.
    Consider the functions $f_1,f_2:\Om\times\rr^2\longrightarrow [0,+\infty)$ defined by 
    \begin{equation*}
     f_1(x,\eta)=2\left(\frac{\eta_1+\eta_2}{2}\right)^2\qquad \text{and}\qquad f_2(x,\eta)=2\left(\frac{\eta_1+\eta_2}{2}\right)^2+e^{(\eta_1-\eta_2)^2}-1
    \end{equation*}
    for any $x\in\Om$ and any $\eta=(\eta_1,\eta_2)\in\rr^2$. They are clearly Carathéodory functions. In view of \eqref{immaginecontroes}, they both verify \eqref{0702i}, \eqref{0702i2} and \eqref{0802conv2} with $a,b=0$ and $c,d=
   1$. Moreover, 
    \begin{equation}\label{ugualidiverse}
        f_1(x,\C(x)\cdot\xi)=f_2(x,\C(x)\cdot\xi)
    \end{equation}
     for any $x\in\Om$ and any $\xi\in\rr^2$, but they differ otherwise. In particular $f_1$ satisfies \eqref{costonker}, while $f_2$ does not.
    Consider the local functionals $F_1,F_2:L^2(\Om)\times\mA\longrightarrow [0,+\infty]$ defined by
\begin{equation*}\label{708funzionaleduedef24}
F_j(u,A)=
\displaystyle{\begin{cases}
\int_{A} f_j(x,Xu(x))\,d x&\text{ if }A\in\mA,\,u\in W^{1,2}_{X,loc}(A)\\
+\infty&\text{otherwise}
\end{cases}
\,.
}
\end{equation*}
for $j=1,2$. Clearly $F_1$ and $F_2$ verify (i), (ii) and (iv) in \Cref{2411modiofmythm24}.  By means of the forthcoming \Cref{sufficiente}, it holds that 
\begin{equation}\label{ugualifinczionali}
    F_1(u,A)=F_2(u,A)=:F(u,A)
\end{equation}
for any $A\in\mA$ and any $u\in L^2(\Om)$. In particular coupling \eqref{ugualifinczionali} with \cite[Lemma 4.14]{MR4108409}, we conclude that $F$ verifies also (iii) and (v) in \Cref{2411modiofmythm24}, with $a=0$ and $c=1$, so that $F$ verifies all the hypotheses of \Cref{2411modiofmythm24}. In addition, $F$ satisfies \eqref{functionalcoercive} with $d=1$. Nevertheless, on the one hand we know by \eqref{ugualidiverse} that the integral representation of $F$ drastically lack uniqueness. On the other hand, neither $f_1(x,\eta)\geq|\eta|^2$ for a.e. $x\in\Om$ and any $\eta\in\rr^2$, nor $f_2(x,\eta)\leq |\eta|^2$ for a.e. $x\in\Om$ and any $\eta\in\rr^2$.
\end{example}
Despite these differences with respect to the \eqref{lic} framework, we show that the structural properties of $f$ that one can derive from an integral representation as in \Cref{2411modiofmythm24} are essentially the only ones relevant for deducing structural properties of the associated functional.
More precisely, the following holds.
\begin{theorem}\label{sufficiente}
    Let $p\in[1,+\infty)$. Let $f:\Om\times\rr\times\rr^m\longrightarrow[0,+\infty]$ be a Carathéodory function. Let $F:L^p(\Om)\times\mA\longrightarrow[0,+\infty]$ be defined by
\begin{equation*}\label{708funzionaleduedef24}
F(u,A)=
\displaystyle{\begin{cases}
\int_{A} f(x,u(x),Xu(x))\,d x&\text{ if }A\in\mA,\,u\in W^{1,p}_{X,loc}(A)\\
+\infty&\text{otherwise}
\end{cases}
\,.
}
\end{equation*}
The following facts hold. 
    \begin{itemize}
        \item [(i)] If $f$ satisfies \eqref{0702i}, then
  \begin{equation*}
 F(u,A)\leq\int_Aa(x)+b|u(x)|^p+c|Xu|^p\,dx
  \end{equation*}
  for any $A\in\mA$ and any $u\in W^{1,p}_{X,loc}(A)\cap L^p(\Om)$.
  \vspace{3pt}
   \item [(ii)] If $f$ satisfies \eqref{0702i2}, then
  \begin{equation*}
 d\int_A|Xu|^p\,dx\leq F(u,A)
  \end{equation*}
  for any $A\in\mA$ and any $u\in W^{1,p}_{X,loc}(A)\cap L^p(\Om)$.
  \vspace{3pt}
   \item [(iii)] If $\tilde f:\Om\times\rr\times\rr^m\longrightarrow[0,+\infty]$ is another Carathéodory function such that 
   \begin{equation}
       f(x,u,\C(x)\cdot\xi)=\tilde  f(x,u,\C(x)\cdot\xi)
   \end{equation}
for a.e. $x\in\Om$, any $u\in\rr$ and any $\xi\in\rr^n$, then
   \begin{equation*}
       F(u,A)=\int_A \tilde f (x,u(x),Xu(x))\,dx
   \end{equation*}
   for any $A\in\mA$ and any $u\in W^{1,p}_{X,loc}(A)\cap L^p(\Om)$.
    \end{itemize}
\end{theorem}
\begin{proof}
    In view of \eqref{euclrapprgrad}, the three statements are clearly true for any $A\in\mA$ and any $u\in C^\infty(A)\cap L^p(\Om)$. Noticing that all the involved functionals are continuous with respect to the metric topology of $W^{1,p}_X$, the general statement follows by means of standard localization and continuity arguments (cf. \cite{MR4108409,MR4609808}) coupled with the density of $C^\infty\cap W^{1,p}_X$ in $W^{1,p}_X$ with respect to the metric topology of $W_X^{1,p}$ (cf. \cite{MR1437714} and \cite[Proposition 2.8]{MR4108409}).
\end{proof}
\subsection{$\Gamma$-compactness} We conclude this section with the generalization of \cite[Theorem 4.10]{MR4108409}.
\begin{theorem}\label{gammathmwithproof}
   Let $p\in(1,+\infty)$. For any $h\in\mathbb N$, let $f_h:\Om\times\rr^m\longrightarrow [0,+\infty]$ be a Carathéodory function satisfying \eqref{0702i}, \eqref{0702i2} and \eqref{0802conv2} with $a\in L^1(\Om)$, $b=0$, $c\geq 0$ and $d>0$ independent of $h\in\mathbb N$. For any $h\in\mathbb N$, define the integral functional $F_h:L^p(\Om)\times\mA\longrightarrow [0,+\infty]$ by 
    \begin{equation*}\label{708funzionaleduedef24}
F_h(u,A)=
\displaystyle{\begin{cases}
\int_{A} f_h(x,Xu(x))\,d x&\text{ if }A\in\mA,\,u\in W^{1,p}_{X}(A)\\
+\infty&\text{otherwise}
\end{cases}
\,.
}
\end{equation*}
Then, up to a subsequence, there exists an integral functional of the form
    \begin{equation*}\label{708funzionaleduedef24}
F(u,A)=
\displaystyle{\begin{cases}
\int_{A} f(x,Xu(x))\,d x&\text{ if }A\in\mA,\,u\in W^{1,p}_{X}(A)\\
+\infty&\text{otherwise}
\end{cases}
\,,
}
\end{equation*}
where $f:\Om\times\rr^m\longrightarrow [0,+\infty)$ is a Carathéodory function which satisfies \eqref{costonker}, \eqref{0702i}, \eqref{0702i2} and \eqref{0802conv2} with $a,b,c,d$ as above, for which 
\begin{equation}\label{gammaconvstatement}
    F(\cdot,A)=\Gamma(L^p)-\lim_{h\to+\infty}F_{h}(\cdot,A)
\end{equation}
for any $A\in\mA$.
\end{theorem}
\begin{proof}
    By means of (i) and (ii) in \Cref{sufficiente}, it is simply a matter of retracing the steps of the proof of \cite[Proposition 3.3]{MR4566142} to ensure the existence of a functional $F:L^p(\Om)\times\mA\longrightarrow [0,+\infty]$ which verifies (i), (ii), (iii), (v) and \eqref{functionalcoercive} in \Cref{2411modiofmythm24} and such that \eqref{gammaconvstatement} holds. We claim that $F$ verifies (iv) in \Cref{2411modiofmythm24}. 
    To this aim, fix $A\in\mA$, $u\in C^\infty A$ and $k\in\rr$. We only show that $F(u,A)\geq F(u+k,A)$, being the other inequality analogous. If $F(u,A)=+\infty$, the claim is trivial. Assume otherwise that $F(u,A)$ is finite. Let $(u_h)_h\subseteq L^p(\Om)$ be a recovery sequence for $u$ as in \eqref{recodef}. Since $F(u,A)$ is finite, up to a subsequence $(u_h)_h\subseteq W^{1,p}_X(A)\cap L^p(\Om)$. Therefore, by our choice of $(u_h)_h$, \eqref{liminfineq}, \eqref{recodef} and the definition of $(F_h)_h,$
    \begin{equation*}
F(u,A)=\liminf_{h\to+\infty}F_h(u_h,A)=\liminf_{h\to+\infty}F_h(u_h+k,A)\geq F(u+k,A).
    \end{equation*}
    To conclude, $F$ satisfies the hypotheses of \Cref{2411modiofmythm24}, so that 
    there exists a Carathéodory function $f:\Om\times\rr^m\longrightarrow [0,+\infty)$ which satisfies \eqref{costonker}, \eqref{0702i}, \eqref{0702i2} and \eqref{0802conv2} with $a,b,c,d$ as in the statement, such that \eqref{primarappr0802} holds for any $A\in\mA$ and any $u\in W^{1,p}_X(A)\cap L^p(\Om)$. Finally, fix $A\in\mA$ and let $u\in L^p(\Om)\setminus W^{1,p}_X(A)$. If it was the case that $F(u,A)<+\infty$, then $u$ would admit a recovery sequence $(u_h)_h\subseteq W^{1,p}_X(A)\cap L^p(\Om)$. But then, in view of \eqref{recodef}, (ii) of \Cref{sufficiente} and \cite[Lemma 4.14]{MR4108409},
    \begin{equation*}
    F(u,A)=\liminf_{h\to+\infty}F_h(u_h,A)\geq \liminf_{h\to+\infty}d\int_A|Xu_h|^p\,dx\geq d\int_A|Xu|^p\,dx=+\infty,
    \end{equation*}
    from which a contradiction would follow.
\end{proof}
\section{Further statements}\label{furthersection}
For future references, we include in this last section the statements which generalize the corresponding results in \cite{MR4609808,MR4566142} avoiding \eqref{lic}. We omit their proof since, owing to \Cref{anisorappr} and \Cref{sufficiente}, they can be recovered following the original approach of \cite{MR4609808,MR4566142} as done in the previous section. 
\subsection{Integral representation} The following results generalize, respectively, \cite[Theorem 3.3]{MR4609808}, \cite[Theorem 4.3]{MR4609808} and \cite[Theorem 5.6]{MR4609808}. We refer to \cite{MR4609808,MR4566142} for the notation. 
\begin{theorem}\label{ir2}
	Let $p\in[1,+\infty)$. Let $F:L^p(\Om)\times\mathcal{A}\longrightarrow [0,+\infty]$ satisfy the following properties.
	\begin{itemize}
		\item [(i)]$F$ is a measure.
		\item [(ii)]$F$ is local.
		\item [(iii)]$F$ is convex on $W^{1,p}_X(\Om)$.
		\item [(iv)] There exist $a\in L^1_{loc}(\Om)$ and $b,c\geq 0$ such that
  \begin{equation}\label{boundinstatement}
      F(u,A)\leq\int_Aa(x)+b|u|^p+c|Xu|^p\,dx
  \end{equation}
  for any $A\in\mA$ and any $u\in C^\infty(A)\cap L^p(\Om)$.
	\end{itemize}
	Then there exists a Carathéodory function $f:\Omega\times\rr\times\mathbb{R}^m\longrightarrow [0,+\infty)$ 
	such that	
 \begin{equation}\label{primarappr0802otherstatements}
     F(u,A)=\int_{A}f(x,u(x),Xu(x))\,d x
 \end{equation}
 for any $A\in\mA$ and any $u\in W^{1,p}_{X,loc}(A)\cap L^p(\Om)$.
 Moreover, $f$ satisfies \eqref{costonker}, \eqref{0702i} and \eqref{0802conv1}.
In addition, if $\tilde f:\Om\times\rr^m\longrightarrow [0,+\infty)$ is a Carathéodory function which verifies \eqref{costonker}, \eqref{0702i} and for which \eqref{primarappr0802} holds with $\tilde f$ in place of $f$, then
	\begin{equation} \label{2411uniquenessssother}
		\tilde f(x,u,\eta)=f(x,u,\eta)
	\end{equation}
 for a.e. $x\in\Om$, any $u\in\rr$ and any $\eta\in\rr^m$. Finally, if there exists $d>0$ such that  \eqref{functionalcoercive} holds, then $f$ satisfies \eqref{0702g2}.
\end{theorem}
\begin{theorem}\label{ir3}
	Let $p\in[1,+\infty)$. Let $F:L^p(\Om)\times\mathcal{A}\longrightarrow [0,+\infty]$ satisfy the following properties.
	\begin{itemize}
		\item [(i)]$F$ is a measure.
		\item [(ii)]$F$ is local.
  \item [(iii)] $F$ satisfies the weak condition $(\omega)$.
		\item [(iv)]$F$ satisfies \eqref{boundinstatement}.
  \item [(v)] $F$ is lower semicontinuous on $W^{1,p}_X(\Om)$.
  \item [(vi)] $F$ is weakly*-sequentially lower semicontinuous.
	\end{itemize}
	Then there exists a Carathéodory function $f:\Omega\times\rr\times\mathbb{R}^m\longrightarrow [0,+\infty)$ 
	such that \eqref{primarappr0802otherstatements} holds.
 Moreover, $f$ satisfies \eqref{costonker}, \eqref{0702i} and \eqref{0802conv2}.
In addition, $f$ is unique in the sense of \eqref{2411uniquenessssother}. Finally, if there exists $d>0$ such that  \eqref{functionalcoercive} holds, then $f$ satisfies \eqref{0702g2}.
\end{theorem}
\begin{theorem}\label{ir4}
	Let $p\in[1,+\infty)$. Let $F:L^p(\Om)\times\mathcal{A}\longrightarrow [0,+\infty]$ satisfy the following properties.
	\begin{itemize}
		\item [(i)]$F$ is a measure.
		\item [(ii)]$F$ is local.
  \item [(iii)] $F$ satisfies the strong condition $(\omega)$.
		\item [(iv)]$F$ satisfies \eqref{boundinstatement}.
  \item [(v)] $F$ is lower semicontinuous on $W^{1,p}_X(\Om)$.
	\end{itemize}
	Then there exists a Carathéodory function $f:\Omega\times\rr\times\mathbb{R}^m\longrightarrow [0,+\infty)$ 
	such that \eqref{primarappr0802otherstatements} holds.
 Moreover, $f$ satisfies \eqref{costonker} and \eqref{0702i}.
In addition, $f$ is unique in the sense of \eqref{2411uniquenessssother}. Finally, if there exists $d>0$ such that  \eqref{functionalcoercive} holds, then $f$ satisfies \eqref{0702g2}.
\end{theorem}
\subsection{$\Gamma$-compactness}  The following results generalize, respectively, \cite[Theorem 3.1]{MR4566142}, \cite[Theorem 4.3]{MR4566142} and \cite[Theorem 4.4]{MR4566142}. We refer again to \cite{MR4566142} for the notation. 
\begin{theorem}
  Let $p\in(1,+\infty)$. For any $h\in\mathbb N$, let $f_h:\Om\times\rr\times\rr^m\longrightarrow [0,+\infty]$ be a Carathéodory function which satisfies \eqref{0702i}, \eqref{0702i2} and \eqref{0802conv1} with $a\in L^1(\Om)$, $b,c\geq 0$ and $d>0$ independent of $h\in\mathbb N$. For any $h\in\mathbb N$, define the integral functional $F_h:L^p(\Om)\times\mA\longrightarrow [0,+\infty]$ by 
    \begin{equation*}\label{fhintegrother}
F_h(u,A)=
\displaystyle{\begin{cases}
\int_{A} f_h(x,u(x),Xu(x))\,d x&\text{ if }A\in\mA,\,u\in W^{1,p}_{X}(A)\\
+\infty&\text{otherwise}
\end{cases}
\,.
}
\end{equation*}
Then, up to a subsequence, there exists an integral functional of the form
    \begin{equation*}\label{fintegrother}
F(u,A)=
\displaystyle{\begin{cases}
\int_{A} f(x,u(x),Xu(x))\,d x&\text{ if }A\in\mA,\,u\in W^{1,p}_{X}(A)\\
+\infty&\text{otherwise}
\end{cases}
\,,
}
\end{equation*}
where $f:\Om\times\rr\times\rr^m\longrightarrow [0,+\infty)$ is a Carathéodory function which satisfies \eqref{costonker}, \eqref{0702i}, \eqref{0702i2} and \eqref{0802conv1} with $a,b,c,d$ as above, for which \eqref{gammaconvstatement} holds.
\end{theorem}
\begin{theorem}
   Let $p\in[1,+\infty)$. For any $h\in\mathbb N$, let $f_h:\Om\times\rr\times\rr^m\longrightarrow [0,+\infty]$ be a Carathéodory function which satisfies \eqref{0702i} and \eqref{0802conv1} with $a\in L^1(\Om)$ and $b,c\geq 0$ independent of $h\in\mathbb N$. For any $h\in\mathbb N$, define the integral functional $F_h:W^{1,p}_X(\Om)\times\mA\longrightarrow [0,+\infty]$ by 
    \begin{equation}\label{fhintegrother}
F_h(u,A)=
\int_{A} f_h(x,u(x),Xu(x))\,d x 
\end{equation}
for any $A\in\mA$ and any $u\in W^{1,p}_{X}(\Om)$. Then, up to a subsequence, there exists an integral functional of the form
    \begin{equation}\label{fintegrother}
F(u,A)=\int_{A} f(x,u(x),Xu(x))\,d x 
\end{equation}
for any $A\in\mA$ and any $u\in W^{1,p}_{X}(\Om)$,
where $f:\Om\times\rr\times\rr^m\longrightarrow [0,+\infty)$ is a Carathéodory function which satisfies \eqref{costonker}, \eqref{0702i} and \eqref{0802conv1} with $a,b,c$ as above, for which 
\begin{equation}\label{gammaconvstatement2}
    F(\cdot,A)=\Gamma(W^{1,p}_X)-\lim_{h\to+\infty}F_{h}(\cdot,A)
\end{equation}
for any $A\in\mA$.
\end{theorem}
\begin{theorem}
   Let $p\in[1,+\infty)$. For any $h\in\mathbb N$, let $f_h:\Om\times\rr\times\rr^m\longrightarrow [0,+\infty]$ be a Carathéodory function which satisfies \eqref{0702i} with $a\in L^1(\Om)$ and $b,c\geq 0$ independent of $h\in\mathbb N$. For any $h\in\mathbb N$, define the integral functional $F_h:W^{1,p}_X(\Om)\times\mA\longrightarrow [0,+\infty]$ as in \eqref{fhintegrother}. Assume that $(F_h)_h$ satisfies a uniform strong condition $(\omega X)$ with respect to $\omega$ (cf. \cite[Definition 4.1]{MR4566142}). Then, up to a subsequence, there exists an integral functional as in \eqref{fintegrother},
where $f:\Om\times\rr\times\rr^m\longrightarrow [0,+\infty)$ is a Carathéodory function which satisfies \eqref{costonker} and \eqref{0702i} with $a,b,c$ as above, for which \eqref{gammaconvstatement2} holds. Finally, $f$ satisfies the strong condition $(\omega X)$ with respect to $\omega$.
\end{theorem}

\bibliographystyle{abbrv}
\bibliography{biblio}
\end{document}